\documentclass{cimart}

\newtheorem{lemmaArt}{Lemma}[section]

\newtheorem{propArt}[lemmaArt]{Proposition}

\theoremstyle{definition}

\newtheorem{exampleArt}[lemmaArt]{Example}

\newenvironment{eq}{\begin{equation}}{\end{equation}}

\newcommand{\Char}{\mathop{\rm char}}
\newcommand{\id}[1]{{{\rm id}\{{#1}\}}}

\newcommand{\FF}{\mathbb{F}}

\newcommand{\ZZ}{\mathbb{Z}}

\newcommand{\NN}{\mathbb{N}}
\newcommand{\tq}{ \ | \ }
\newcommand{\FX}{\FF\langle X\rangle}
\newcommand{\Fxy}{\FF\langle x,y\rangle}

\newcommand{\A}{\mathsf{A}}
\newcommand{\B}{\mathsf{B}}

\newcommand{\W}{\mathsf{W}}

\newcommand{\algA}{\mathcal{A}}
\newcommand{\algB}{\mathcal{B}}

\newcommand{\algV}{\mathcal{V}}

\newcommand{\LA}{\langle}
\newcommand{\RA}{\rangle}
\newcommand{\eqPI}{\sim_{\rm PI}}

\newcommand{\si}{\sigma}
\newcommand{\al}{\alpha}
\newcommand{\be}{\beta}
\newcommand{\ga}{\gamma}

\newcommand{\de}{\delta}
\newcommand{\De}{\Delta}
\newcommand{\Ga}{\Gamma}
\newcommand{\La}{\Lambda}
\newcommand{\Id}[1]{{{\rm Id}({#1})}}
\newcommand{\IdF}[1]{{{\rm Id}_{\FF}({#1})}}

\newcommand{\un}[1]{{\underline{#1}} }

\newcommand{\mdeg}{\mathop{\rm mdeg}}

\newcommand{\lin}{\mathop{\rm lin}}

\newcommand{\St}{{\rm St}}

\title{
    Identities for subspaces of the Weyl algebra
    }

\author{
    Artem Lopatin, Carlos Arturo Rodriguez Palma
    }

\authorinfo[
    A. Lopatin]{
    Universidade Estadual de Campinas (UNICAMP), SPbU}{%
    dr.artem.lopatin@gmail.com
    }

\authorinfo[
    C.A. Rodriguez Palma]{
    Universidad Industrial de Santander, Cra 27 calle 9, Ciudad Universit\'aria, Bucaramanga, Santander, Colombia}{%
    caropal@uis.edu.co
    }

\abstract{%
    In this paper we describe the polynomial identities of degree 4 for a certain subspace of the Weyl algebra over an infinite field of arbitrary characteristic.  
    }



\VOLUME{32}
\YEAR{2024}
\ISSUE{2}
\firstpage{111}
\DOI{https://doi.org/10.46298/cm.12323}

\begin{document}

\section{Introduction}\label{section_intro}

Assume that $\FF$ is an infinite field of arbitrary characteristic $p=\Char\FF\geq0$. All vector spaces and algebras are over $\FF$  and all algebras are unital and associative, unless stated otherwise. We write $\FF\LA x_1,\ldots,x_n\RA$ for the free unital $\FF$-algebra with free generators  $x_1,\ldots,x_n$. In case  the set of free generators is infinite and enumerable, and denoted by $X=\{x_1,x_2,\dots\}$, the corresponding free algebra is denoted by $\FF\LA X\RA$.

\subsection{Witt algebra $\W_1$}


The {\it Weyl algebra} $\A_1$ is the unital associative algebra over $\FF$ generated by letters $x$, $y$ subject to the defining relation $yx=xy+1$ (equivalently, $[y,x]=1$, where $[y,x]=yx-xy$), i.e., $$\A_1=\Fxy/\id{yx-xy-1}.$$

\noindent{}For $s>0$ define by $\A_1^{(-,s)}$ the $\FF$-span of $a y^s$ in $\A_1$ for all $a\in \FF[x]$. It is easy to see that the following two conditions hold:
\begin{enumerate}
\item[$\bullet$] the space $\A_1^{(-,s)}$ is closed with respect to the Lie bracket $[\,\cdot,\cdot\,]$;

\item[$\bullet$]  the Lie bracket $[\,\cdot,\cdot\,]$ is not trivially zero on $\A_1^{(-,s)}$ if and only if $s=1$ (for example, see Corollary 3.5 of~\cite{Lopatin_Rodriguez_II}).
\end{enumerate}
 Note that in case $p=0$ the space $\A_1^{(-,1)}$ together with the multiplication given by the Lie bracket is the Witt algebra $\W_1$, which is a simple infinite dimensional Lie algebra. Similarly, considering $n$  pairs $\{x_i,y_i\}$ $(1\leq i\leq n)$ instead of $\{x,y\}$ we can define  the $n^{\rm th}$ Witt algebra $\W_n$, which is also a simple infinite dimensional Lie algebra.

\subsection{Polynomial identities}\label{section_intro_PI} 

A polynomial identity for a unital $\FF$-algebra $\algA$ is an element $f(x_1,\ldots,x_m)$ of $\FF\LA X\RA$ such that $f(a_1,\ldots,a_m)=0$ in $\algA$ for all $a_1,\ldots,a_m\in \algA$.   The set $\IdF{\algA}=\Id{\algA}$ of all polynomial identities for $\algA$ is a T-ideal, that is,  $\Id{\algA}$ is an ideal of $\FX$ such that $\phi(\Id{\algA})\subset \Id{\algA}$ for every endomorphism  $\phi$ of $\FX$.  An algebra that satisfies a nontrivial polynomial identity is called a PI-algebra. A T-ideal $I$ of $\FF\LA X\RA$ generated polynomials $f_1,\ldots,f_k$ in $\FF\LA X\RA$ is the minimal T-ideal of $\FF\LA X\RA$ that contains $f_1,\ldots,f_k$. Denote $I=\Id{f_1,\ldots,f_k}$. We say that $f\in \FF\LA X\RA$  is a consequence of $f_1,\ldots,f_k$ if $f\in I$. Given a monomial $w$ in $\FF\LA x_1,\ldots,x_m\RA$, we write $\deg_{x_i}(w)$ for the number of letters $x_i$ in $w$ and $\mdeg(w)\in\NN^m$ for the multidegree $(\deg_{x_1}(w),\ldots,\deg_{x_m}(w))$ of $w$, where $\NN=\{0,1,2,\ldots\}$. An element $f\in\FX$ is called (multi)homogeneous if it is a linear combination of monomials of the same (multi)degree. Given $f=f(x_1,\ldots,x_m)$ of $\FX$, we write $f=\sum_{\un{\de}\in \NN^m} f_{\un{\de}}$ for multihomogeneous components $f_{\un{\de}}$ of $f$ with multidegree $\mdeg{f_{\un{\de}}}=\un{\de}$.  For $\un{\de}=(\de_1,\ldots,\de_m)$ we denote $|\un{\de}|=\de_1+\cdots+\de_m$. We say that algebras $\algA$, $\algB$ are PI-equivalent and write $\algA \eqPI \algB$ if $\Id{\algA} =\Id{\algB}$.

Given an $\FF$-subspace $\algV\subset \algA$,  we write $\IdF{\algV}=\Id{\algV}$ for the ideal of all polynomial identities for $\algV$. Note that  $\phi(\Id{\algV})\subset \Id{\algV}$ for every linear endomorphism  $\phi$ of $\FX$, but $\Id{\algV}$ is not a T-ideal in general. 

Assume that $p=0$. It is well-known that the algebra $\A_1$ does not have nontrivial polynomial identities. Namely, it follows from  Kaplansky's Theorem~\cite{Kaplansky} and  the fact that $\A_1$ is simple with ${\rm Z}(\A_1)=\FF$.
Nevertheless, some subspaces of $\A_1$ satisfy certain polynomial identities. As an example, Dzhumadil'daev proved that the standard polynomial $${\rm St}_N(x_{1},\ldots,x_{N})=\sum_{\sigma\in S_{N}}(-1)^{\sigma}x_{\sigma(1)}\cdots x_{\sigma(N)}$$
is a polynomial identity for $\A_1^{(-,s)}$ if and only if $N>2s$ (Theorem 1 of~\cite{Askar_2014}). More results on polynomial identities for some subspaces of $n^{\rm th}$ Weyl algebra were obtained in~\cite{Askar_2004, Askar_Yeliussizov_2015}.  The polynomial Lie identities for the $n^{\rm th}$ Witt algebra $\W_n$ were studied by Mishchenko~\cite{Mishchenko_1989}, Razmyslov~\cite{Razmyslov_book} and others. The well-known open conjecture claims that all polynomial identities for $\W_1$ follow from the standard Lie identity
$$\sum_{\sigma\in S_{4}}(-1)^{\sigma}[[[[x_0,x_{\si(1)}],x_{\si(2)}],x_{\si(3)}],x_{\si(4)}].$$
\noindent{}$\ZZ$-graded identities for $W_1$ were described by Freitas, Koshlukov and  Krasilnikov~\cite{W1_2015}. Moreover, $\ZZ$-graded identities for the related Lie algebra of the derivations of the algebra of Laurent polynomials were described in~\cite{Fideles_Koshlukov_2023_JA, Fideles_Koshlukov_2023_Camb}.  

The situation is drastically different in case $p>0$. Namely, $\A_1$ is PI-equivalent to the algebra $M_p$ of all $p\times p$ matrices over $\FF$. Moreover,  the Weyl algebra $\A_1$ over an arbitrary associative (but possible non-commutative) $\FF$-algebra $\B$ is PI-equivalent to the algebra $M_p(\B)$ of all $p\times p$ matrices over $\B$ (see Theorem~4.9  of~\cite{Lopatin_Rodriguez_2022} for more general result). 

Over a field of an arbitrary characteristic, minimal polynomials identities for 
\begin{enumerate}
\item[$\bullet$] $\A_1^{(-,1)}$ for an arbitrary $p$,
\item[$\bullet$] $\A_1^{(-,s)}$ for $p=2$,  
\end{enumerate}
were described in~\cite{Lopatin_Rodriguez_II}. Moreover, similar result was obtained in~\cite{Lopatin_Rodriguez_II} for the so-called parametric Weyl algebras, which were introduced and studied by Benkart, Lopes, Ondrus~\cite{Benkart_Lopes_Ondrus_II, Benkart_Lopes_Ondrus_I,  Benkart_Lopes_Ondrus_III}.

\subsection{Results} In this paper we described all polynomial identities for $\A_1^{(-,1)}$ of degree 4. Namely, polynomial identities of multidegree
\begin{enumerate}
\item[$\bullet$] (3,1) were considered in Proposition~\ref{prop_deg31}; 

\item[$\bullet$] (2,2) were considered in Proposition~\ref{prop_deg22}; 

\item[$\bullet$] (2,1,1) were considered in Proposition~\ref{prop_id211};

\item[$\bullet$] (1,1,1,1) were considered in Propositions~\ref{prop_id1111} and~\ref{prop_id1111_p2}.
\end{enumerate}%
\noindent{}It is clear that there is no polynomial identities of multidegree $(4)$. See~\cite{comp} for the computer program for Wolfram Mathematica to assist the proofs of  Lemma~\ref{lemma_deg22_id},~\ref{lemma_deg1111_id} and Propositions~\ref{prop_id1111},~\ref{prop_id1111_p2}.

\section{Auxiliary notions}  

\subsection{Properties of $\A_1$}\label{section_A1}

Given $a\in \FF[x]$, we write $a'$ for the usual derivative of the polynomial $a$ with respect to the variable $x$. Using the linearity of derivative and induction on the degree of $a\in\FF[x]$ it is easy to see that
\begin{eq}\label{eq0}
[y,a]=a' \text{ holds in }\A_1 \text{ for all }a\in \FF[x].
\end{eq}%
Thus for all $i,j\geq0$ the associative multiplication and the Lie bracket on $\A_1^{(-,1)}$ are given by
\begin{eq}\label{eq1}
x^i y\, x^j y = x^{i+j}y^2 + j x^{i+j-1}y \;\text{ and }\; [x^i y,x^j y] = (j-i)x^{i+j-1}y,
\end{eq}%
where we use notation that
$$x^i=0 \text{ in }\A_1 \text{ in case }i\in\ZZ \text{ is negative}.$$
The following properties are well-known (for example, see~\cite{Benkart_Lopes_Ondrus_I}):

\begin{propArt}\label{prop_basis1}  
\begin{enumerate}
\item[(a)] $\{x^{i}y^{j}\tq i,j\geq 0\}$ and $\{y^{j}x^{i}\tq i,j\geq 0\}$ are $\FF$-bases for $\A_1$. In particular,  $\{x^{i}y\tq i\geq 0\}$ is an $\FF$-basis for $\A_1^{(-,1)}$.

\item[(b)] If $p=0$, then the center ${\rm Z}(\A_1)$ of $\A_1$ is $\FF$; if $p>0$, then ${\rm Z}(\A_1)=\FF[x^{p},y^{p}]$.

\item[(c)] If $p>0$, then $\A_1$ is a free module over ${\rm Z}(\A_1)$ and the set $\{x^{i}y^{j} \tq 0\leq i,j<p\}$ is a basis.

\item[(d)] The algebra $\A_1$ is simple if and only if $p=0$.
\end{enumerate}
\end{propArt}
\medskip

Theorem 5.4 of~\cite{Lopatin_Rodriguez_II} implies the following statement:

\begin{propArt}\label{theo_PI_min} 
\begin{enumerate}
	\item[(a)] A minimal polynomial identity for $\A_1^{(-,1)}$ has degree $3$.
	
	\item[(b)] Every homogeneous polynomial identity for  $\A_1^{(-,1)}$  of degree $3$ is equal to $\xi\, \St_3$ for some $\xi\in\FF$.
\end{enumerate}
\end{propArt}
\medskip

\subsection{Partial linearizations}

Assume $f\in \FX$ is multihomogeneous of multidegree $\un{\de}\in\NN^m$. Given $1\leq i\leq m$ and $\un{\ga}\in\NN^k$ for some $k>0$ with $|\un{\ga}|=\de_i$, the {\it partial linearization} $\lin_{x_i}^{\un{\ga}}(f)$ of $f$ of multidegree $\un{\ga}$ with respect to $x_i$ is the multihomogeneous component of 
$$f(x_1,\ldots,x_{i-1},x_{i}+\cdots+x_{i+k-1},x_{i+k},\ldots,x_{m+k-1})$$ 
of multidegree $(\de_1,\ldots,\de_{i-1},\ga_{1},\ldots,\ga_{k},\de_{i+1},\ldots,\de_{m})$. As an example, $$\rm{lin}_{x_2}^{(1,1)}(x_1x_2^2x_3^3) = x_1 (x_2 x_3 + x_3 x_2)x_4^3.$$
The result of subsequent applications of partial linearizations to $f$ is also called a partial linearization of $f$. The {\it complete linearization} $\lin(f)$ of $f$ is the result of subsequent applications of $\lin_{x_1}^{1^{\de_1}},\ldots, \lin_{x_m}^{1^{\de_m}}$ to $f$, where $1^k$ stands for $(1,\ldots,1)$ ($k$ times). 

Assume $\algA$  is a unital $\FF$-algebra  and  $\algV\subset \algA$ is an $\FF$-subspace. 
Since $\FF$ is infinite, it is well-known that if $f$ is a polynomial identity for $\algV$, then all partial linearizations of $f$ are also polynomial identities for $\algV$. Note that the above claim does not hold in general for a finite field (as an example, see~\cite{Lopatin_Shestakov_2013} for the case of $f(x_1)=x_1^n$ and $\algV=\algA$). The following lemma is well-known.

\begin{lemmaArt}\label{lemma_id} Assume that all partial linearizations of a multihomogeneous element  $f$ of  $\FX$ are equal to zero over some basis of $\algV$. Then $f$ is a polynomial identity for $\algV$. 
\end{lemmaArt}
\begin{proof}
Let $\{v_j\}$ be a basis for $\algV$. Note that $f(\sum_j \al_{1j} v_j, \ldots, \sum_j \al_{mj} v_j)$ is a linear combination of partial linearizations of $f$ evaluated on the basis $\{v_j\}$, where $\al_{1j},\ldots,\al_{mj}\in\FF$ for all $j$. Therefore, the required is proven. 
\end{proof}

\begin{exampleArt} Assume that $p=2$ and $\algA$ is the unital associative commutative algebra generated by $e_1,\ldots,e_n$ with the ideal of relations generated by $e_1^2,\ldots,e_n^2$, where $n>0$. Denote by $\algV$ the maximal ideal of $\algA$ generated by $e_1,\ldots,e_n$. Let us apply Lemma~\ref{lemma_id} to show that $f(x_1)=x_1^2$ is the polynomial identity for $\algV$. 

All partial linearizations of $f(x_1)$ are $f(x_1)$ and $f_{11}(x_1,x_2)=x_1x_2 + x_2x_1$. Since $B=\{e_{i_1}\cdots e_{i_k}\,|\,1\leq i_1<\cdots<i_k\leq n,\; k\geq1\}$ is a basis for $\algV$ and 
$$f(e_{i_1}\cdots e_{i_k})=f_{11}(e_{i_1}\cdots e_{i_k},e_{j_1}\cdots e_{j_r})=0$$ 
in $\algA$ for all $1\leq i_1<\cdots<i_k\leq n$, $1\leq j_1<\cdots<j_r\leq n$ with $k,r>0$, we obtain that $f(x_1)\in\Id{\algV}$ by Lemma~\ref{lemma_id}. Note that $f(x_1)$ is not a polynomial identity for $\algA$, since $f(1)=1$, but $f_{11}\in\Id{\algA}$. 
\end{exampleArt}

\section{Non-multilinear identities for $\A_1^{(-,1)}$ of degree four}\label{section_deg4} 

In this section we will denote $c_i=x^i y$ for $i\geq0$. Note that in the presentation $c_i c_j c_k c_l=\sum_{r,t\geq0}\be_{rt} x^r y^t$ in $\A_1$ the coefficient $\be_{rt}\in\FF$ is unique for all $r,t\geq0$ by part (a) of Proposition~\ref{prop_basis1}.

Consider the following multihomogeneous elements of $\FF\LA X\RA$ of degree $4$:
$$\Phi_{22}= x_1^2 x_2^2 - 3 x_1 x_2 x_1 x_2 + 2 x_1 x_2^2 x_1 + 2 x_2 x_1^2 x_2 - 3 x_2 x_1 x_2 x_1 + x_2^2 x_1^2,$$
$$\Psi = x_2[x_1,x_4]x_3 + x_3[x_1,x_4]x_2,$$
$$\Psi_{211} = \Psi(x_1,x_1,x_3,x_2) = x_1[x_1,x_2]x_3 + x_3[x_1,x_2]x_1 .$$
Denote  $\Phi_{211} =  \rm{lin}_{x_2}^{(1,1)}\Phi_{22}$, $\Phi_{\rm lin} =  \rm{lin } (\Phi_{22})$ and  $$\Psi_{\rm lin} =  \rm{lin} (\Psi_{211}) = \Psi(x_1,x_2,x_4,x_3) + \Psi(x_2,x_1,x_4,x_3).$$

\begin{lemmaArt}\label{lemma_decomposition}
Given  $i,j,k,l\geq0$, we denote $e=i+j+k+l$. Then in the presentation of $c_i c_j c_k c_l$ as the linear combination of basis elements $\{x^r y^t\,|\,r,t\geq0\}$ of $\A_1$ the coefficient of
\begin{enumerate}
\item[(a)] $x^e y^4$ is $1$;

\item[(b)] $x^{e-1} y^3$ is $j+2k+3l$, in case $e\geq1$;

\item[(c)] $x^{e-2} y^2$ is $(k+2l)(j+k+l-1) + l(k+l-1)$, in case $e\geq2$;

\item[(d)] $x^{e-3}y$ is $l(k+l-1)(j+k+l-2)$, in case $e\geq3$.
\end{enumerate}
The remaining coefficients are zeros. Moreover, we may apply parts (b), (c), (d) for every $e\geq0$, since in case of negative degree of $x$ the corresponding coefficient is zero.    
\end{lemmaArt}
\begin{proof} Assume $i,j,k,l\geq1$. We apply equality~(\ref{eq0}) to obtain that
$$\begin{array}{rcl}
c_jc_kc_l & = & x^j y x^k (x^l y + l x^{l-1}) y \\
& = &  x^j( y x^{k+l}) y^2 + l x^j ( y x^{k+l-1})y\\
& = &  x^{j+k+l} y^3 + (k+2l) x^{j+k+l-1} y^2 +  l (k+l-1) x^{j+k+l-2} y \;\;\;\text{ in }\;\;\A_1.
\end{array}
$$
Applying the obtained formula to $c_i c_j c_k c_l$ we  conclude the proof.  Note that in these calculations $x^k$ with negative $k\in\ZZ$ always has zero coefficient. Then these calculation are valid for $i,j,k,l\geq0$ and the required is proven.
\end{proof}

 \begin{propArt}\label{prop_deg31} There is no non-trivial multihomogeneous polynomial identities for $\A_1^{(-,1)}$ of multidegree $(3,1)$.
\end{propArt}
\begin{proof} Assume that $f(x_1,x_2) = \al_1 x_1^3 x_2 + \al_2 x_1^2 x_2 x_1 + \al_3 x_1 x_2 x_1^2 + \al_4 x_2 x_1^3$ is a polynomial identity for $\A_1^{(-,1)}$, where $\al_1,\ldots,\al_4\in\FF$. We will show that $f=0$ is the trivial identity. For all $i,j\geq1$ we have $f(c_i,c_j)=0$ in $\A_1$. Applying parts (a)--(d), respectively, of Lemma~\ref{lemma_decomposition}, we obtain that 
\begin{eq}\label{eqA}
\al_1+\al_2+\al_3+\al_4=0,
\end{eq}
\begin{eq}\label{eqB}
(3i+3j)\al_1 + (4i+2j)\al_2 + (5i+j)\al_3 + 6i\al_4 = 0,
\end{eq}
\begin{eq}\label{eqC}
\begin{array}{c}
(2i^2 + 6ij + 3j^2 - i - 3j)\al_1 + (5i^2 + 5ij + j^2 -3i - j)\al_2 +\\ 
(8i^2 + 3ij - 4i)\al_3 + (11i^2-4i)\al_4 = 0,\\
\end{array}
\end{eq}
\begin{eq}\label{eqD}
\begin{array}{c}
j(i+j-1)(2i+j-2)\al_1 + i(i+j-1)(2i+j-2)\al_2 + \\
i(2i-1)(2i+j-2)\al_3 + i(2i-1)(3i-2)\al_4 = 0,\\
\end{array}
\end{eq}%
respectively. We can rewrite formula~(\ref{eqB}) as 
\begin{eq}\label{eqB2}
(3\al_1+4\al_2+5\al_3+6\al_4)i + (3\al_1+ 2\al_2 +\al_3)j = 0.
\end{eq}%

We subtract equality~(\ref{eqB2}) with $i=j=1$ from equality~(\ref{eqB2}) with $i=1$, $j=2$ to obtain that  $3\al_1+ 2\al_2 +\al_3=0$.  Thus it follows from equality~(\ref{eqA}) that for an arbitrary $p$ we have $\al_3=-3\al_1 - 2\al_2$ and $\al_4=2\al_1+\al_2$.  Taking $i=1$, $j=2$ and $i=1$, $j=3$ in equality~(\ref{eqD}) we obtain that $\al_2=-4\al_1=0$. Thus $f=0$ in case $p\neq2$.

Assume $p=2$. We have $\al_2=\al_4=0$ and $\al_3=\al_1$. Considering $i=1$, $j=2$ in equality~(\ref{eqC}) we can see that $\al_1=0$. The required is proven.
\end{proof}

\begin{lemmaArt}\label{lemma_deg22_id} The elements $\Phi_{22}$, $\Phi_{211}$, $\Phi_{\rm lin}$ are polynomial identities for $\A_1^{(-,1)}$. In case $p=2$ the elements $\Psi_{211}$, $\Psi$, and $[[x_1,x_2],[x_3,x_4]]$ are  polynomial identities for $\A_1^{(-,1)}$.
\end{lemmaArt}
\begin{proof}Using Lemma~\ref{lemma_decomposition} and straightforward calculations (by means of a computer program) we can see that $\Phi_{22}$ and its partial linearizations $\lin_{x_1}^{(1,1)}(\Phi_{22})$, $\Phi_{211}=\lin_{x_2}^{(1,1)}(\Phi_{22})$ and the complete linearization $\Phi_{\rm lin}=\lin(\Phi_{22})$ are equal to zero over the set $\{c_i\, |\, i\geq0\}$. Since $\{c_i\, |\, i\geq0\}$ is a basis of $\A_1^{(-,1)}$, Lemma~\ref{lemma_id} concludes the proof for $\Phi_{22}$, $\Phi_{211}$, $\Phi_{\rm lin}$. 

Assume $p=2$. Since $\Psi_{211}$, $\Psi$, $\Psi_{\rm lin}$, and $[[x_1,x_2],[x_3,x_4]]$ are zero over the set $\{c_i\, |\, i\geq0\}$, Lemma~\ref{lemma_id} concludes the proof for $\Psi_{211}$, $\Psi$, and $[[x_1,x_2],[x_3,x_4]]$. 

\end{proof}

\begin{propArt}\label{prop_deg22} Assume  $f$ is a multihomogeneous polynomial identity for $\A_1^{(-,1)}$ of multidegree $(2,2)$. Then $f=\al\, \Phi_{22}$ for some $\al\in \FF$.
\end{propArt}
\begin{proof}  Assume that $$f(x_1,x_2) = \al_1 x_1^2 x_2^2 + \al_2 x_1 x_2 x_1 x_2 +\al_3 x_1 x_2^2 x_1 + \al_4  x_2 x_1^2 x_2 + \al_5 x_2 x_1 x_2 x_1 + \al_6 x_2^2 x_1^2$$ is a polynomial identity for $\A_1^{(-,1)}$, where $\al_1,\ldots,\al_6\in\FF$.  
Hence $f(c_i,c_j)=0$ in $\A_1$ for all $i,j\geq0$. Applying parts (a)--(d), respectively, of Lemma~\ref{lemma_decomposition}, we obtain that 
\begin{eq}\label{eq22A}
\al_1+\al_2+\al_3+\al_4+\al_5+\al_6=0,
\end{eq}
\begin{eq}\label{eq22B}
\begin{array}{c}
(i+5j)\al_1 + (2i+4j)\al_2 + (3i+3j)\al_3 + \\
(3i+3j)\al_4 + (4i+2j)\al_5 + (5i+j)\al_6=0, \;\;\text{ if }\;\; i+j\geq1, \\
\end{array}
\end{eq}
\begin{eq}\label{eq22C}
\begin{array}{c}
(3ij + 8j^2 - 4j)\al_1 + (i^2 + 5ij + 5j^2 -i - 3j)\al_2 + \\
(3i^2 + 6ij + 2j^2 - 3i - j)\al_3 + (2i^2 + 6ij + 3j^2 - i - 3j)\al_4 + \\
(5i^2 + 5ij + j^2 - 3i - j)\al_5 + (8i^2 + 3 ij - 4i)\al_6=0, \;\;\text{ if }\;\; i+j\geq1, \\
\end{array}
\end{eq}
\begin{eq}\label{eq22D}
\begin{array}{c}
j(2j-1)(i+2j-2)\al_1 + j(i+j-1)(i+2j-2)\al_2 + \\
i(i+j-1)(i+2j-2)\al_3 + j(i+j-1)(2i+j-2)\al_4 + \\
i(i+j-1)(2i+j-2)\al_5 + i(2i-1)(2i+j-2)\al_6=0, \;\;\text{ if }\;\; i+j\geq2.\\
\end{array}
\end{eq}%
Taking $i=0$, $j=1$ in equality~(\ref{eq22B}) we obtain 
$$5\al_1 + 4\al_2 + 3\al_3 + 3\al_4 + 2\al_5 + \al_6=0.$$
Considering $i=0$, $j=1$ and $i=1$, $j=0$ in equality~(\ref{eq22C}) we obtain that 
\begin{eq}\label{eq_new1}
4\al_1 + 2\al_2 +\al_3 = 0\quad \text{and}\quad \al_4 + 2\al_5 + 4\al_6 = 0,
\end{eq}%
respectively. 

Let $p\neq2$.  Considering $i=0$, $j=2$ and $i=2$, $j=0$ in equality~(\ref{eq22D}) we obtain that $3\al_1 + \al_2 = 0$ and $\al_5 + 3\al_6=0$, respectively.  It it easy to see that the above five equalities imply that 
\begin{eq}\label{eq22g1}
\al_2=\al_5=-3\al_1,\;\; \al_3=\al_4=2\al_1 \;\text{ and }\; \al_6=\al_1.
\end{eq}%
Hence, $f=\al_1 \Phi_{22}$. 

Let $p=2$. Equality (\ref{eq22B}) implies that $(\al_1+\al_3+\al_4+\al_6)(i+j)=0$  for all $i,j\geq0$ with $i+j\geq1$. Considering $i=1$, $j=0$ we obtain that $\al_1+\al_3+\al_4+\al_6=0$.  Similarly, equality~(\ref{eq22C}) implies that  $ij(\al_1+\al_2+\al_5+\al_6) + j\al_3 + i\al_4 =0$ for all $i,j\geq0$ with $i+j\geq1$. Equalities~(\ref{eq_new1}) imply that $\al_3=\al_4=0$. Applying equality~(\ref{eq22A}) we can see that   
\begin{eq}\label{eq22g2}
\al_3=\al_4=0, \;\; \al_5=\al_2 \;\text{ and }\; \al_6=\al_1.
\end{eq}%
In other words, $f=\al_1 [x_1^2,x_2^2] + \al_2 (x_1 x_2 x_1 x_2 + x_2 x_1 x_2 x_1)$. By Lemma~\ref{lemma_decomposition} we can see that
\begin{eq}\label{eq_lin}
0=\lin\nolimits_{x_2}^{(1,1)}(f)(c_1,c_1,c_2)=(\al_1+\al_2) x^2y.
\end{eq}
Thus $\al_1=\al_2$ and the proof is completed.
\end{proof} 

Recall that for $i_1,\ldots,i_k,j_1,\ldots,j_k\in\ZZ$, two {\it multisets} $\{i_1,\ldots,i_k\}_{\rm m}$ and $\{j_1,\ldots,j_k\}_{\rm m}$ are equal if for every $l\in\ZZ$ we have $|\{ 1\leq t\leq k \tq i_t=l \}| = |\{ 1\leq t\leq k \tq j_t=l \}|$.

\begin{lemmaArt}\label{lemma_deg211} Assume $\algA$ is an associative algebra and $\algV\subset\algA$ is an $\FF$-subspace. Suppose
\begin{enumerate}
\item[(a)] any polynomial identity of $\algV$ of multidegree $(3,1)$ is trivial;

\item[(b)] any polynomial identity of $\algV$ of multidegree $(2,2)$ is equal to  $\xi \Phi_{22}$ for some $\xi\in \FF$.
\end{enumerate} Then every polynomial identity $f$ of $\algV$ of multidegree $(2,1,1)$ is equal to 
$$\al x_1 \St_3(x_1,x_2,x_3) -\be\, [[x_1,x_2],[x_1,x_3]] + \ga\, \St_3(x_1,x_2,x_3)x_1 + \xi h(x_1,x_2,x_3)$$
for some $\al,\be,\ga,\xi\in\FF$ and  $h(x_1,x_2,x_3)$ is given as
$$ h(x_1,x_2,x_3) =x_1^2 x_3 x_2 + 2 x_3 x_1^2 x_2 + x_3 x_2 x_1^2  - x_1 x_2 x_3 x_1  + 3 x_1 x_3 x_2 x_1 -3 x_1 x_3 x_1 x_2  - 3 x_3 x_1 x_2 x_1.$$
\end{lemmaArt}
\begin{proof}
First, we have that $f(x_1,x_2,x_3)=\sum {\al_{ijkl}}x_i x_j x_k x_l$, where the sum ranges over all $1\leq i,j,k,l\leq 3$ with $\{i,j,k,l\}_{\rm m}=\{1,1,2,3\}_{\rm m}$ and $\al_{ijkl}\in\FF$. For short, we write $\al_{1^223}$ for $\al_{1123}$, etc. Applying part (a) to $f(x_1,x_1,x_2)=0$ we obtain that
$$\begin{array}{rcl}
\al_{1^223} + \al_{21^23} + \al_{1213}
& = &  0,\\
\al_{231^2} + \al_{1312} + \al_{1321} & = &  0,\\
\al_{31^22} + \al_{321^2} + \al_{3121} & = &  0.\\
\end{array}$$
\noindent{}Similarly, applying part (b) to $f(x_1,x_2,x_2)=0$ we obtain that
$$\begin{array}{rcl}
\al_{1^223} + \al_{1^232} & = & \xi, \\ 
\al_{231^2} + \al_{321^2} & = & \xi,\\
\al_{1213} + \al_{1312} & = & -3\xi,\\
\al_{2131} + \al_{3121} & = & -3\xi,\\
\al_{1231} + \al_{1321} & = & 2\xi,\\
\al_{21^23} + \al_{31^22} & = & 2\xi.\\
\end{array}$$
These equations imply the required equality for $f$ for $\al = \al_{1^223}$, $\be=\al_{21^23}$, $\ga=\al_{231^2}$.
\end{proof}


\begin{propArt}\label{prop_id211} The following set is an $\FF$-basis of the space of all polynomial identities for $\A_1^{(-,1)}$ of multidegree $(2,1,1)$:
\begin{enumerate}
\item[(a)] $x_1 \St_3(x_1,x_2,x_3)$, $\St_3(x_1,x_2,x_3)x_1$, $\Phi_{211}$, in case $p\neq 2$; 
\item[(b)] $x_1 \St_3(x_1,x_2,x_3)$, $\St_3(x_1,x_2,x_3)x_1$, $\Psi_{211}$, $[[x_1,x_2],[x_1,x_3]]$, in case $p= 2$.
\end{enumerate}
\end{propArt}
\begin{proof} By Proposition~\ref{theo_PI_min} and Lemma~\ref{lemma_deg22_id} all elements from the formulation of the proposition are identities for $\A_1^{(-,1)}$.  

\medskip
\noindent{\bf 1.} At first, we will show that any polynomial identity $f\in\FX$ for $\A_1^{(-,1)}$ of multidegree $(2,1,1)$ is an $\FF$-linear combination of elements from the formulation of the proposition. Since $x_1 \St_3(x_1,x_2,x_3)$ and  $\St_3(x_1,x_2,x_3)x_1$ are polynomial identities for $\A_1^{(-,1)}$ by Lemma~\ref{theo_PI_min}, Lemma~\ref{lemma_deg211} implies that it is enough to complete the proof for 
$$f(x_1,x_2,x_3) = -\be\, [[x_1,x_2],[x_1,x_3]] +  \xi h(x_1,x_2,x_3),$$
where $\be,\xi\in\FF$. 

Assume $p\neq2$. Since $0=f(c_3,c_2,c_1) = 2 (\be-\xi)x^{6} y$ by Lemma~\ref{lemma_decomposition}, we obtain that $\xi=\be$.  By straightforward calculations we can see that
$$ [[x_1,x_2],[x_1,x_3]] - h(x_1,x_2,x_3) = \frac{1}{2}\left(x_1 \St_3(x_1,x_2,x_3) + \St_3(x_1,x_2,x_3)x_1 - \Phi_{211}\right).$$
The required is proven.

Assume that $p=2$. By straightforward calculations we can see that
$$h(x_1,x_2,x_3)= x_1 \St_3(x_1,x_2,x_3) + \Psi_{211}(x_1,x_2,x_3).$$ 
Note that $$\Phi_{211}= x_1 \St_3(x_1,x_2,x_3) + \St_3(x_1,x_2,x_3) x_1.$$ 
The required is proven.

\medskip
\noindent{\bf 2.} To show that  elements from the formulation of the lemma are linearly independent in case $p\neq2$, we consider 
$$\al\, x_1 \St_3(x_1,x_2,x_3) + \be\, \St_3(x_1,x_2,x_3)x_1 + \ga\, \Phi_{211}(x_1,x_2,x_3) = 0$$
for some $\al,\be,\ga\in\FF$. Taking the coefficients of $x_1^2x_2x_3$ and $x_1^2x_3x_2$ we therefore obtain $\al+\ga=-\al+\ga=0$. Thus $\al=\ga=0$ and the required is proven.

Similarly, for $p=2$ we consider 
$$\al\, x_1 \St_3(x_1,x_2,x_3) + \be\, \St_3(x_1,x_2,x_3)x_1 + \ga\, \Psi_{211} +\de [[x_1,x_2],[x_1,x_3]] =0$$
for some $\al,\be,\ga,\de\in\FF$. Taking the coefficients of $x_1^2x_3x_2$, $x_2x_3x_1^2$, $x_1x_3x_2x_1$ we obtain $-\al=\be=-\al-\be+\de=0$. Thus $\al=\be=\de=0$ and the required is proven.

\end{proof}

\section{Multilinear identities for $\A_1^{(-,1)}$ of degree four}\label{section_mult4} 

As in Section~\ref{section_deg4} we denote $c_i=x^i y$ for $i\geq0$. Consider the following multilinear elements of $\FF\LA X\RA$ of degree $4$:
$$\Ga=   -x_ 1 x_2 x_3 x_4 + 2 x_1 x_2 x_4 x_3 + x_1 x_3 x_4 x_2 - 
   2 x_1 x_4 x_2 x_3 $$
   $$ + 2 x_2 x_1 x_3 x_4 - 2 x_2 x_1 x_4 x_3 - 2 x_2 x_3 x_1 x_4 + x_2 x_3 x_4 x_1 + x_2 x_4 x_1 x_3 $$
   $$+ x_3 x_1 x_2 x_4 - 2 x_3 x_1 x_4 x_2 + x_3 x_4 x_1 x_2 + x_4 x_1 x_2 x_3 - x_4 x_2 x_3 x_1, $$
\vspace{-0.2cm}  
$$\La =  -3 x_1 x_2 x_3 x_4 + 3 x_1 x_2 x_4 x_3 + 2 x_1 x_3 x_2 x_4 - 2 x_1 x_4 x_2 x_3 $$
$$+ 3 x_2 x_1 x_3 x_4 - 3 x_2 x_1 x_4 x_3 - 2 x_2 x_3 x_1 x_4 + 2 x_2 x_4 x_1 x_3 $$
$$- x_3 x_1 x_4 x_2 +  x_3 x_2 x_4 x_1 + x_4 x_1 x_3 x_2 - x_4 x_2 x_3 x_1,$$
\vspace{-0.2cm}  
$$\De=x_2x_1x_3x_4 + x_2x_4x_1x_3 + x_3x_1x_2x_4 + x_3x_4x_1x_2 + x_4x_1x_2x_3 + x_4x_1x_3x_2.$$

A monomial from $\FF\LA X\RA$ of multidegree $(1,1,1,1)$ is called {\it reduced} if it does not belong to the following list:
\begin{eq}\label{eq_non_reduced}
x_1x_4x_3x_2,\; x_2x_4x_3x_1,\;  x_3x_2x_1x_4, \; x_3x_4x_2x_1, \; 
x_4x_2x_1x_3,\; x_4x_3x_1x_2,\; x_4x_3x_2x_1.
\end{eq}%
An element from $\FF\LA X\RA$ of multidegree $(1,1,1,1)$ is called {\it reduced} if it is a linear combination of reduced monomials. Note that $\Ga$, $\La$, $\De$ are reduced.

\begin{lemmaArt}\label{lemma_ireduced}
For every homogeneous $f\in\FF\LA X\RA$ of multidegree $(1,1,1,1)$ there exist multilinear $f_1,f_2\in \FF\LA X\RA$ of degree 4 such that $f=f_1+f_2$,
\begin{enumerate}
\item[$\bullet$] $f_1$ is reduced;

\item[$\bullet$] $f_2$ is a linear combination of polynomials of the form $x_i \St_3(x_j,x_k,x_l)$, $\St_3(x_i,x_j,x_k)x_l$ where $\{i,j,k,l\}=\{1,2,3,4\}$.
\end{enumerate}
\end{lemmaArt}
\begin{proof}
Consider the usual lexicographical order on the set of all monomials from $\FF\LA X\RA$
of multidegree $(1,1,1,1)$. Denote by $L$ the subspace of $\FF\LA X\RA$ generated by  $x_i \St_3(x_j,x_k,x_l)$, $\St_3(x_i,x_j,x_k)x_l$ for $\{i,j,k,l\}=\{1,2,3,4\}$. Given a monomial $w\in\FF\LA X\RA$ of multidegree $(1,1,1,1)$, we write $w\equiv 0$, if $w-h\in L$ for some $h\in\FF\LA X\RA$ such that all monomials of $h$ are less than $w$. Since $x_1\St_3(x_2,x_3,x_4)\in L$, we obtain that $x_1  x_4 x_3 x_2\equiv0$. Similarly, considering $x_2\St_3(x_1,x_3,x_4)\in L$, $x_3\St_3(x_1,x_2,x_4)\in L$, $x_4\St_3(x_1,x_2,x_3)\in L$, respectively,  we obtain that 
$$x_2x_4x_3x_1\equiv0,\;\; x_3x_4x_2x_1\equiv0, \;\; x_4x_3x_2x_1\equiv0,$$
respectively. Moreover, considering  $\St_3(x_1,x_3,x_4)x_2\in L$, $\St_3(x_1,x_2,x_3)x_3\in L$ and also $\St_3(x_1,x_2,x_3)x_4\in L$, respectively, we can see that 
$$x_4x_3x_1x_2\equiv 0,\;\; x_4x_2x_1x_3\equiv0,\;\; x_3x_2x_1x_4\equiv0,$$
respectively. Consequently, applying the obtained equivalences to $f$, it is easy to see that the claim holds. 
\end{proof}

\begin{lemmaArt}\label{lemma_deg1111_id} The elements $\Ga$,  $\La$ are polynomial identities for $\A_1^{(-,1)}$. If  $p=2$, then $\De$ is also a polynomial identity for $\A_1^{(-,1)}$. 
\end{lemmaArt}
\begin{proof}Using Lemma~\ref{lemma_decomposition} and straightforward calculations (by means of a computer program) we can see that $\Ga$, $\La$ are equal to zero over the set $\{c_i\, |\, i\geq0\}$. Since the set  $\{c_i\, |\, i\geq0\}$ is a basis of $\A_1^{(-,1)}$, Lemma~\ref{lemma_id} concludes the proof. Similarly, we prove Lemma~\ref{lemma_deg1111_id} for $\De$ in case $p=2$. 
\end{proof}

The following remark can be verified by straightforward calculations.

\begin{remark}\label{remark_equality} The following equalities hold in $\FF\LA X\RA$:
\begin{enumerate}
\item[1.] \begin{align*}
4\, \Ga - 2\,\La = \ & \Phi_{\rm lin} + 
x_1 \St_3(x_2,x_3,x_4) + x_2\St_3(x_1,x_3,x_4) + 2\, x_3 \St_3(x_1,x_2,x_4)\\  
 &+\St_3(x_2,x_3,x_4)x_1 + \St_3(x_1,x_3,x_4)x_2 + 2\, \St_3(x_1,x_2,x_4)x_3.  
\end{align*}

\item[2.] \begin{align*}
& \ x_1 \St_3(x_2,x_3,x_4) - x_2 \St_3(x_1,x_3,x_4) + x_3 \St_3(x_1,x_2,x_4) - x_4 \St_3(x_1,x_2,x_3) \\ 
& + \St_3(x_2,x_3,x_4)x_1 - \St_3(x_1,x_3,x_4)x_2 + \St_3(x_1,x_2,x_4)x_3 - \St_3(x_1,x_2,x_3)x_4 = 0. 
\end{align*}
\end{enumerate}
\end{remark}

\begin{propArt}\label{prop_id1111} The following set is an $\FF$-basis of the space of all polynomial identities for $\A_1^{(-,1)}$ of multidegree $(1,1,1,1)$ in case $p\neq 2$:
$$\Ga,\; \Phi_{\rm lin}, \; x_1\St_3(x_2,x_3,x_4), \; x_2\St_3(x_1,x_3,x_4), \; x_3\St_3(x_1,x_2,x_4), \; x_4\St_3(x_1,x_2,x_3), $$
$$\St_3(x_2,x_3,x_4)x_1,\; \St_3(x_1,x_3,x_4)x_2,\; \St_3(x_1,x_2,x_4)x_3.$$
\end{propArt}
\begin{proof} By Proposition~\ref{theo_PI_min} and Lemmas~\ref{lemma_deg22_id},~\ref{lemma_deg1111_id} all elements from the formulation of the proposition are identities for $\A_1^{(-,1)}$. By part 1 of Remark~\ref{remark_equality} we can consider $\La$ instead of  $\Phi_{\rm lin}$ in the formulation of the proposition.

Assume that $f\in\FX$ is a polynomial identity for $\A_1^{(-,1)}$ of multidegree $(1,1,1,1)$.  By Lemma~\ref{lemma_ireduced} and part 2 of Remark~\ref{remark_equality} we can assume that $f$ is reduced, i.e., 

\begin{eq}\label{eq_reduced}
\begin{array}{rl}
f = & \hspace{.4cm} \al_1\, x_1 x_2 x_3 x_4  +  \al_2\, x_1 x_2 x_4 x_3 + 
   \al_3\, x_1 x_3 x_2 x_4  + \al_4\,x_1 x_3 x_4 x_2 + 
   \al_5\, x_1 x_4 x_2 x_3 \\
&+\, \al_6\, x_2 x_1 x_3 x_4 + \al_7\, x_2 x_1 x_4 x_3 + \al_8\, x_2 x_3 x_1 x_4 + \al_9\, x_2 x_3 x_4 x_1 + \al_{10}\, x_2 x_4 x_1 x_3 \\
 & + \al_{11}\, x_3 x_1 x_2 x_4 +  \al_{12}\, x_3 x_1 x_4 x_2+ \al_{13}\, x_3 x_2 x_4 x_1  + \al_{14}\, x_3 x_4 x_1 x_2 \\
 &  + \al_{15}\, x_4 x_1 x_2 x_3 + \al_{16}\, x_4 x_1 x_3 x_2 + \al_{17}\, x_4 x_2 x_3 x_1 ,\\
   \end{array}
\end{eq}
\noindent{}where $\al_1,\ldots,\al_{17} \in \FF$. Note that 
$$\Ga=  h_1 - x_4 x_2 x_3 x_1  \;\text{ and }\; \La = h_2 + x_4 x_1 x_3 x_2 -x_4 x_2 x_3 x_1,$$
where $h_1,h_2$ are linear combinations of reduced monomials different from  $x_4 x_1 x_3 x_2$ and $x_4 x_2 x_3 x_1$. Then $h= f  + (\al_{16} + \al_{17})\Ga - \al_{16} \La$ does not contain monomials $x_4 x_1 x_3 x_2$, $x_4 x_2 x_3 x_1$. Hence, considering polynomial identity $h$ instead of $f$, we can assume that $\al_{16}=\al_{17}=0$.

To obtain equations on $\al_1,\ldots,\al_{15}$ we consider $f(c_i,c_j,c_k,c_l)=0$ and take the coefficient of $x^r y^s$ for certain $i,j,k,l,r,s$. The resulting linear equation $\ga_1\al_1 +\cdots +\ga_{15} \al_{15}=0$ for some $\ga_1,\ldots,\ga_{15}\in\FF$ we write down as the line $(\ga_1,\ldots,\ga_{15})$ in the matrix $A$ below. Here is the list of parameters  $i,j,k,l,r,s$ which we consider:

\begin{center}
    \begin{tabular}{l l l}
        $\bullet$ $f(c_1,c_0,c_0,c_0)$, $xy^4$;
            & $\bullet$ $f(c_1,c_0,c_0,c_0)$, $y^3$;
            & $\bullet$ $f(c_0,c_1,c_0,c_0)$, $y^3$; \\
        $\bullet$ $f(c_0,c_0,c_1,c_0)$, $y^3$;
            & $\bullet$ $f(c_2,c_0,c_0,c_0)$, $y^2$; 
            & $\bullet$ $f(c_0,c_2,c_0,c_0)$, $y^2$; \\
        $\bullet$ $f(c_0,c_0,c_2,c_0)$, $y^2$; 
            & $\bullet$ $f(c_1,c_1,c_0,c_0)$, $y^2$;
            & $\bullet$ $f(c_1,c_0,c_1,c_0)$, $y^2$; \\
        $\bullet$ $f(c_1,c_1,c_1,c_0)$, $y$; 
            & $\bullet$ $f(c_1,c_1,c_0,c_1)$, $y$; 
            &$\bullet$ $f(c_1,c_0,c_1,c_1)$, $y$; \\
        $\bullet$ $f(c_2,c_1,c_0,c_0)$, $y$; 
            & $\bullet$ $f(c_2,c_0,c_1,c_0)$, $y$; 
            & $\bullet$ $f(c_0,c_2,c_1,c_0)$, $y$. 
    \end{tabular}
\end{center}















The resulting matrix is 
$$A=\left(
\begin{array}{ccccccccccccccc}
1 &	1 & 1 &	1 &	1 & 1 &	1 &	1 &	1 &	1 &	1 &	1 &	1 &	1 &	1 \\
0 &	0 &	0 &	0 &	0 &	1 &	1 &	2 &	3 &	2 &	1 &	1 &	3 &	2 &	1 \\
1 &	1 &	2 &	3 &	2 &	0 &	0 &	0 &	0 &	0 &	2 &	3 &	1 &	3 &	2 \\
2 &	3 &	1 &	1 &	3 &	2 &	3 &	1 &	1 &	3 &	0 &	0 &	0 &	0 &	3 \\
0 &	0 &	0 &	0 &	0 &	0 &	0 &	2 &	6 &	2 &	0 &	0 &	6 &	2 &	0 \\
0 &	0 &	2 &	6 &	2 &	0 &	0 &	0 &	0 &	0 &	2 &	6 &	0 &	6 &	2 \\
2 &	6 &	0 &	0 &	6 &	2 &	6 &	0 &	0 &	6 &	0 &	0 &	0 &	0 &	6 \\
0 &	0 &	0 &	0 &	0 &	0 &	0 &	0 &	0 &	0 &	1 &	2 &	2 &	4 &	1 \\
0 &	0 &	0 &	0 &	0 &	1 &	2 &	1 &	2 &	4 &	0 &	0 &	0 &	0 &	2 \\
0 &	0 &	0 &	0 &	0 &	0 &	0 &	0 &	0 &	0 &	0 &	0 &	0 &	0 &	1 \\
0 &	0 &	0 &	0 &	0 &	0 &	0 &	0 &	0 &	0 &	1 &	1 &	1 &	1 &	0 \\
0 &	0 &	0 &	0 &	0 &	1 &	1 &	1 &	1 &	1 &	0 &	0 &	0 &	0 &	0 \\
0 &	0 &	0 &	0 &	0 &	0 &	0 &	0 &	0 &	0 &	0 &	0 &	2 &	2 &	0 \\
0 &	0 &	0 &	0 &	0 &	0 &	0 &	0 &	2 &	2 &	0 &	0 &	0 &	0 &	0 \\
0 &	0 &	0 &	2 &	2 &	0 &	0 &	0 &	0 &	0 &	0 &	0 &	0 &	0 &	2 \\
\end{array}
\right).$$
Since $\det(A)=-64$ is non-zero, we obtain that $\al_1=\cdots=\al_{15}=0$, i.e, $f=0$. Thus any  polynomial identity $f\in\FX$ for $\A_1^{(-,1)}$ of multidegree $(1,1,1,1)$ is an $\FF$-linear combination of polynomial identities from the formulation of the proposition. 

Note that we have proven that any element $f\in\FX$ of multidegree $(1,1,1,1)$ can be written as a linear combination of 15 monomials, modulo the subspace generated by elements from the formulation of the proposition.  Comparing the dimensions, we obtain that elements from the formulation of the proposition are linearly independent.
\end{proof}

\begin{propArt}\label{prop_id1111_p2} The following set is an $\FF$-basis of the space of all polynomial identities for $\A_1^{(-,1)}$ of multidegree $(1,1,1,1)$ in case $p=2$:
$$\Ga,\; \Psi,\;\De,\;\La, \; x_1\St_3(x_2,x_3,x_4), \; x_2\St_3(x_1,x_3,x_4), \; x_3\St_3(x_1,x_2,x_4), \; x_4\St_3(x_1,x_2,x_3), $$
$$\St_3(x_2,x_3,x_4)x_1,\; \St_3(x_1,x_3,x_4)x_2,\; \St_3(x_1,x_2,x_4)x_3,$$
$$[[x_1,x_3],[x_2,x_4]].$$
\end{propArt}
\begin{proof} By Proposition~\ref{theo_PI_min} and Lemmas~\ref{lemma_deg22_id},~\ref{lemma_deg1111_id} all elements from the formulation of the proposition are identities for $\A_1^{(-,1)}$. 

Assume that $f\in\FX$ is a polynomial identity for $\A_1^{(-,1)}$ of multidegree $(1,1,1,1)$. By Lemma~\ref{lemma_ireduced} and part 2 of Remark~\ref{remark_equality} we can assume that $f$ is reduced, i.e., can be written as in formula~(\ref{eq_reduced}). Note that  
\begin{align*}
g = & \ [[x_1,x_3],[x_2,x_4]] + x_2\St_3(x_1,x_3,x_4) + \St_3(x_1,x_2,x_4)x_3 \\
= & \ x_1x_2x_4x_3 + x_1x_3x_2x_4 + x_1x_3x_4x_2 + x_1x_4x_2x_3 \\ 
+& \ x_2x_1x_3x_4 + x_2x_3x_1x_4 + x_2x_3x_4x_1 + x_2x_4x_1x_3  \\
+& \ x_3x_1x_2x_4 + x_3x_1x_4x_2 + x_4x_1x_2x_3 + x_4x_2x_3x_1.
\end{align*}
is reduced. Thus
\begin{eq}\label{eq_S}
\begin{array}{rcl}
\Ga  &=& h_1 + x_3x_4x_1x_2 + x_4 x_2 x_3 x_1, \\ 
\Psi &=& h_2 + x_3x_4x_1x_2\\ 
\De  &=& h_3 + x_3x_4x_1x_2 + x_4x_1x_3x_2,\\
\La  &=& h_4 + x_1x_2x_4x_3 + x_3x_2x_4x_1+x_4 x_1 x_3 x_2 + x_4 x_2 x_3 x_1,\\
g    &=& h_5 + x_1x_2x_4x_3 + x_4x_2x_3x_1,\\
\end{array}
\end{eq}


\noindent{}where $h_1,\ldots,h_5$ are linear combinations of reduced monomials which do not lie in the set $S$:
$$S=\{x_1x_2x_4x_3,\; x_3x_2x_4x_1,\; x_3x_4x_1x_2,\;x_4 x_1 x_3 x_2, x_4 x_2 x_3 x_1\}.$$ 

Consider~(\ref{eq_S}) as a system of liner equations on elements of $S$. Since the corresponding matrix
$$
\left(
\begin{array}{ccccc}
0 & 0 & 1 & 0 & 1 \\
0 & 0 & 1 & 0 & 0 \\
0 & 0 & 1 & 1 & 0 \\
1 & 1 & 0 & 1 & 1 \\
1 & 0 & 0 & 0 & 1 \\
\end{array}
\right)
$$
is invertible, the elements  from the set $S$ are linear combinations of $\Ga$, $\Psi$, $\De$, $\La$, $g$ and reduced monomials which do not lie in the set $S$. 
Hence, without loss of generality we can assume that $\al_2=\al_{13}=\al_{14}=\al_{16}=\al_{17}=0$.

To obtain equations on $\al_1,\al_3,\al_4,\al_5,\ldots,\al_{12},\al_{15}$ we consider $f(c_i,c_j,c_k,c_l)=0$ and take the coefficient of $x^r y^s$ for certain $i,j,k,l,r,s$. The resulting linear equation $$\ga_1\al_1 +\ga_3\al_3+ \cdots +\ga_{12} \al_{12}+\ga_{15}\al_{15}=0$$ 
for some $\ga_i\in\FF$ we write down as the line  $(\ga_1,\ga_3,\ga_4,\ga_5,\ldots,\ga_{12},\ga_{15})$ in the matrix $A$ below. Here is the list of parameters  $i,j,k,l,r,s$ which we consider:
\begin{center}
    \begin{tabular}{l l l}
$\bullet$ $f(c_1,c_0,c_0,c_0)$, $xy^4$; & $\bullet$ $f(c_1,c_0,c_0,c_0)$, $y^3$; & $\bullet$ $f(c_0,c_1,c_0,c_0)$, $y^3$; \\ 
$\bullet$ $f(c_0,c_0,c_1,c_0)$, $y^3$; & 
$\bullet$ $f(c_1,c_1,c_0,c_0)$, $y^2$; & $\bullet$ $f(c_1,c_0,c_1,c_0)$, $y^2$; \\ 
$\bullet$ $f(c_1,c_1,c_1,c_0)$, $y$; & $\bullet$ $f(c_1,c_1,c_1,c_0)$, $xy^2$; &
$\bullet$ $f(c_1,c_1,c_0,c_1)$, $y$; \\
$\bullet$ $f(c_1,c_1,c_0,c_1)$, $xy^2$; & $\bullet$ $f(c_1,c_0,c_1,c_1)$, $y$; & $\bullet$ $f(c_1,c_0,c_1,c_1)$, $xy^2$.
    \end{tabular}
\end{center}












The resulting matrix is 
$$A=\left(
\begin{array}{cccccccccccc}
1 &	1 &	1 &	1 &	1 &	1 &	1 &	1 &	1 &	1 &	1 &	1 \\
0 &	0 &	0 &	0 &	1 &	1 &	0 &	1 &	0 &	1 &	1 &	1 \\
1 &	0 &	1 &	0 &	0 &	0 &	0 &	0 &	0 &	0 &	1 &	0 \\
0 &	1 &	1 &	1 &	0 &	1 &	1 &	1 &	1 &	0 &	0 &	1 \\
0 &	0 &	0 &	0 &	0 &	0 &	0 &	0 &	0 &	1 &	0 &	1 \\
0 &	0 &	0 &	0 &	1 &	0 &	1 &	0 &	0 &	0 &	0 &	0 \\
0 &	0 &	0 &	0 &	0 &	0 &	0 &	0 &	0 &	0 &	0 &	1 \\
1 &	1 &	0 &	0 &	1 &	0 &	1 &	0 &	0 &	1 &	0 &	1 \\
0 &	0 &	0 &	0 &	0 &	0 &	0 &	0 &	0 &	1 &	1 &	0 \\
0 &	0 &	0 &	1 &	0 &	1 &	0 &	0 &	1 &	1 &	1 &	1 \\
0 &	0 & 0 &	0 &	1 &	1 &	1 &	1 &	1 &	0 &	0 &	0 \\
0 &	0 &	1 &	0 &	1 &	1 &	1 &	1 &	1 &	0 &	1 &	0 \\
\end{array}
\right).$$
Since $\det(A)=1$, we obtain that $\al_1=\al_3=\al_4=\al_5=\cdots=\al_{12}=\al_{15}=0$, i.e, $f=0$. Thus any  polynomial identity $f\in\FX$ for $\A_1^{(-,1)}$ of multidegree $(1,1,1,1)$ is an $\FF$-linear combination of polynomial identities from the formulation of the proposition. 

Note that we have proven that any element $f\in\FX$ of multidegree $(1,1,1,1)$ can be written as a linear combination of 12 monomials, modulo the subspace generated by elements from the from formulation of the proposition.  Comparing the dimensions, we obtain that elements from the formulation of the proposition are linearly independent.
\end{proof}

\section*{Acknowledgments}  This work was supported by RSF 22-11-00081. The authors thank the anonymous referees for valuable comments and suggestions.


\begin{thebibliography}{99}

\bibitem{Benkart_Lopes_Ondrus_II} G. Benkart, S. Lopes, and M. Ondrus. A parametric family of subalgebras of the
Weyl algebra II. Irreducible modules. {\it Contemp. Math.}, 602:73–98, 2013.


\bibitem{Benkart_Lopes_Ondrus_III} G. Benkart, S. Lopes, and M. Ondrus. Derivations of a parametric family of subalgebras of the Weyl algebra. {\it J. Algebra}, 424:46–97, 2015.

\bibitem{Benkart_Lopes_Ondrus_I} G. Benkart, S. Lopes, and M. Ondrus. A parametric family of subalgebras of the
Weyl algebra I. Structure and automorphisms. {\it Trans. AMS}, 367(3):1993–2021, 2015.

\bibitem{Askar_2004} A. Dzhumadil’daev. $N$-commutators. {\it Comment. Math. Helv.}, 79(3):516–553, 2004.

\bibitem{Askar_2014} A. Dzhumadil’daev. $2p$-commutator on differential operators of order $p$. {\it Lett. Math.
Phys.}, 104(7):849–869, 2014.

\bibitem{Askar_Yeliussizov_2015} A. Dzhumadil’daev and D. Yeliussizov. Path decompositions of digraphs and their
applications to Weyl algebra. {\it Adv. in Appl. Math.}, 67:36–54, 2015.

\bibitem{Fideles_Koshlukov_2023_JA} C. Fideles and P. Koshlukov. $\mathbb{Z}$-graded identities of the Lie algebras $U_1$. {\it J. Algebra}, 633:668–695, 2023.

\bibitem{Fideles_Koshlukov_2023_Camb} C. Fidelis and P. Koshlukov. $\mathbb{Z}$-graded identities of the Lie algebras $U_1$ in characteristic 2. {\it Math. Proc. Camb. Phil. Soc.}, 174(1):49–58, 2023.

\bibitem{W1_2015} J. Freitas, P. Koshlukov, and A. Krasilnikov. $\mathbb{Z}$-graded identities of the Lie algebra $W_1$. {\it J. Algebra}, 427:226–251, 2015.

\bibitem{Kaplansky} I. Kaplansky. Rings with a polynomial identity. {\it Bull. Amer. Math. Soc.}, 54:575–580,
1948.

\bibitem{comp} A. Lopatin, \texttt{https://github.com/drartemlopatin/identities-for-subspaces-\\of-weyl-algebra}, 2023.

\bibitem{Lopatin_Rodriguez_2022} A. Lopatin and  C.A. Rodriguez Palma. Identities for a parametric Weyl algebra over a ring. {\it J.
Algebra}, 595:279–296, 2022.

\bibitem{Lopatin_Rodriguez_II} A. Lopatin and  C.A. Rodriguez Palma. Identities for subspaces of a parametric Weyl algebra.
{\it Lin. Algebra Appl.}, 654:250–266, 2022.

\bibitem{Lopatin_Shestakov_2013} A. Lopatin and I. Shestakov. Associative nil-algebras over finite fields. {\it Inter. J.
Algebra Comput.}, 23(8):1881–1894, 2013.

\bibitem{Mishchenko_1989} S. Mishchenko. Solvable subvarieties of a variety generated by a Witt algebra. {\it Math.
USSR Sb.}, 64(2):415–426, 1989.

\bibitem{Razmyslov_book} Yu. Razmyslov. {\it Identities of algebras and their representations}, volume 138 of {\it Transl.
Math. Monogr.} Amer. Math. Soc., Providence, RI, 1994.

\end{thebibliography}


\EditInfo{September 25, 2023}{February 26, 2024}{Adam Chapman, Ivan Kaygorodov, Mohamed Elhamdadi}

\end{document}